\documentclass[12pt,a4paper]{amsart}
\usepackage{a4wide,fullpage,setspace}

\usepackage[T1]{fontenc}
\usepackage{lmodern}
\usepackage{microtype}
\usepackage{mathtools,amssymb}
\usepackage{tikz-cd}
\usepackage[numbers,sort&compress]{natbib}
\definecolor{dark-red}{rgb}{0.5,0.15,0.15}
\usepackage[colorlinks=true,linkcolor=black,citecolor=dark-red,urlcolor=dark-red]{hyperref}
\swapnumbers

\makeatletter
\let\P\@undefined
\let\leq\@undefined
\let\geq\@undefined
\let\vec\@undefined
\let\phi\@undefined
\let\epsilon\@undefined
\makeatother
\newcommand{\geq}{\geqslant}
\newcommand{\vec}{\overrightarrow}
\newcommand{\leq}{\leqslant}
\newcommand{\phi}{\varphi}
\newcommand{\epsilon}{\varepsilon}

\newcommand{\Top}{\mathop{\mathsf{Top}}}
\newcommand{\TOP}{\mathop{\mathsf{TOP}}\nolimits}
\newcommand{\Pcat}{\mathcal P}
\newcommand{\G}{\mathcal G}
\newcommand{\M}{\mathcal M}

\newcommand{\Path}{\mathbb P^{\mathrm{top}}}
\newcommand{\Glob}{\operatorname{Glob}^{\Pcat}}

\newcommand{\cocartesian}{\arrow[lu, phantom, "\ulcorner"{font=\Large}, pos=0]}

\newtheorem{theorem}{Theorem}[section]

\newtheorem{lemma}[theorem]{Lemma}

\theoremstyle{definition}

\newtheorem{remark}[theorem]{Remark}

\title{The q-model category of multipointed d-spaces is not left proper}
\author[P. Gaucher]{Philippe Gaucher}
\address{Universit\'e Paris Cit\'e, CNRS, IRIF, F-75013, Paris, France}
\urladdr{\url{https://www.irif.fr/~gaucher}}

\begin{document}

\begin{abstract}
	We prove that the q-model category of \(\Pcat\)-multipointed \(d\)-spaces for \(\Pcat\in\{\G,\M\}\) is not left proper by constructing a weak equivalence whose pushout along a q-cofibration obtained by attaching a single \(1\)-dimensional globular cell is not a weak equivalence.  The counterexample is constructed in the category of \(\Delta\)-Hausdorff \(\Delta\)-generated spaces.  For \(\Pcat=\G\), all objects are automatically saturated, and for \(\Pcat=\M\), the original weak equivalence is between saturated objects. The failure is caused by a family of nonconstant execution paths that converges in the ambient mapping space to a constant map which is not an execution path; after the globular cell is attached, this degeneration causes two previously distinct path components to merge.
\end{abstract}

\maketitle
\setcounter{tocdepth}{1}
\tableofcontents
\hypersetup{linkcolor = dark-red}

\section{Introduction}

Multipointed \(d\)-spaces provide a topological model of directed systems: their distinguished points represent states, and their execution paths represent admissible evolutions between states.  For either interval reparametrization category \(\Pcat=\G\) or \(\Pcat=\M\), the q-model structure organizes their homotopy theory by declaring a map to be a weak equivalence when it induces a bijection on states and a weak homotopy equivalence on the space of execution paths; see \cite{Moore3}.  A natural stability question is whether this model structure is left proper.  Left properness requires the pushout of a weak equivalence along a cofibration to remain a weak equivalence and therefore controls, in particular, the homotopy invariance of globular cell attachments.

The analogous q-model category of flows is left proper \cite{leftproperflow}. Multipointed \(d\)-spaces, however, have an additional point-set topological feature: their execution paths are actual maps into an underlying topological space, and their topology is inherited from an ambient mapping space.  Forming a colimit can consequently create new limiting behavior among execution paths.  More precisely, a family of nonconstant old execution paths may converge in the ambient mapping space to a constant map which is not itself an execution path. After a new path has been attached, composition with this family can produce a path in one execution-path space joining two paths which lie in distinct components of the other. 

We show that this phenomenon already occurs when a single \(1\)-dimensional globular cell is attached. For each \(\Pcat\in\{\G,\M\}\), we construct a weak equivalence \(s: A\to B\) and a cobase change \(j: A\to X\) of the generating q-cofibration
\[
\Glob(\varnothing)\longrightarrow
\Glob(\mathsf{D}^{0})
\]
such that the pushout \(\overline{s}: X\to Y\) of \(s\) along \(j\) is not a weak equivalence.  Since \(j\) is a q-cofibration, it follows that the q-model category of \(\Pcat\)-multipointed \(d\)-spaces is not left proper. This does not contradict left properness of the Quillen equivalent q-model category of flows since left properness is not invariant under Quillen equivalence. The counterexample takes place in the category of \(\Delta\)-Hausdorff \(\Delta\)-generated spaces, and the failure is already detected on \(\pi_0\) of the execution-path spaces.  In particular, it does not involve the state-identification generator \(R\).  Moreover, for \(\Pcat=\M\), both objects \(A\) and \(B\) are saturated in the sense of \cite[Definition~18]{Moore3}, so saturation of the input objects does not prevent the phenomenon (see Remark~\ref{rem:saturated}).

Section~\ref{sec:2} recalls the conventions and the facts about multipointed \(d\)-spaces used in the proof.  Section~\ref{sec:3} constructs the weak equivalence \(s: A\to B\).  Section~\ref{sec:4} attaches the globular cell and proves that the resulting map fails to be a weak equivalence.

\section{Conventions}
\label{sec:2}

Let \(I=[0,1]\), and let \(\Top\) be the category of \(\Delta\)-Hausdorff \(\Delta\)-generated spaces.  It is cartesian closed, and the mapping space \(\TOP(-,-)\) is obtained by applying \(\Delta\)-kelleyfication to the compact-open topology; see \cite[Section~2 and Appendix~B]{leftproperflow} and \cite[Section~2]{Moore3}.  We work with either \(\Pcat=\G\), whose morphisms are the nondecreasing homeomorphisms between nontrivial compact intervals, or \(\Pcat=\M\), whose morphisms are the nondecreasing surjections between such intervals.  In particular, every element of \(\Pcat(1,1)\) is surjective and \(\G(1,1)\subset\M(1,1)\).  These categories and their topologies are recalled in \cite[Section~2]{Moore3}.

The normalized composition of two continuous paths \(x,y:I\to U\) with \(x(1)=y(0)\) is the continuous path \(x*_Ny:I\to U\) such that \((x*_Ny)(t)=x(2t)\) for \(0\leq t\leq 1/2\) and \((x*_Ny)(t)=y(2t-1)\) for \(1/2\leq t\leq 1\).

A \(\Pcat\)-multipointed \(d\)-space \(X\) consists of a topological space \(\lvert X\rvert\), called the underlying topological space, a set of states \(X^0\subset\lvert X\rvert\), and a set \(\Path X\) of continuous paths \(I\to\lvert X\rvert\), called execution paths, whose endpoints are states.  The set \(\Path X\) is closed under precomposition by \(\Pcat(1,1)\) and under normalized composition.  For \(u,v\in X^0\), let \(\Path_{u,v}X\) denote the set of execution paths from \(u\) to \(v\).  It is equipped with the \(\Delta\)-kelleyfication of the subspace topology induced by
\[
\Path_{u,v}X\subset\TOP(I,\lvert X\rvert).
\]
The total execution-path space is the coproduct of these spaces over \(u,v\in X^0\); see \cite[Section~3]{Moore3}.

Let \(Z\in\Top\). The \(\Pcat\)-multipointed \(d\)-space \(\Glob(Z)\) is defined as follows. Its underlying space is the quotient in \(\Top\)
	\[
	\lvert\Glob(Z)\rvert
	=
	\frac{\{0,1\}\sqcup\bigl(Z\times[0,1]\bigr)}
	{(z,0)\sim 0,\ (z,1)\sim 1\quad(z\in Z)}.
	\]
	Its set of states is
	\[
	\Glob(Z)^0=\{0,1\}.
	\]
	For every \(z\in Z\), let
	\[
	\delta_z:[0,1]\longrightarrow\lvert\Glob(Z)\rvert,
	\qquad
	\delta_z(t)=[z,t].
	\]
	The set of execution paths is
	\[
	\mathbb P^{\mathrm{top}}\Glob(Z)
	=
	\left\{
	\delta_z\circ\phi
	\;\middle|\;
	z\in Z,\ \phi\in\Pcat(1,1)
	\right\}.
	\]
Thus every execution path goes from the initial state \(0\) to the final state \(1\). 

We shall use two facts about the q-model structure.  A map \(f: X\to Y\) is a weak equivalence if and only if \(f^0: X^0\to Y^0\) is a bijection and \(\Path f:\Path X\to\Path Y\) is a weak homotopy equivalence.  A set of generating q-cofibrations consists of the maps
\begin{equation} \label{eq:generator}
\Glob(\mathsf S^{n-1})
\longrightarrow\Glob(\mathsf D^n)
\qquad(n\geq0),
\end{equation}
with $\mathsf{S}^n$ the $n$-dimensional sphere, $\mathsf{D}^n$ the $n$-dimensional disk, \(\mathsf{S}^{-1}=\varnothing\) and \(\mathsf{D}^0=\{*\}\), together with
\[
C:\varnothing\longrightarrow\{0\},
\qquad
R:\{0,1\}\longrightarrow\{0\}.
\]
This is recalled in \cite[Section~4]{Moore3}.  Finally, the forgetful functor to multipointed spaces is topological, and the execution paths in a colimit are exactly the finite Moore compositions of the images of execution paths in the entries of the diagram, with arbitrary positive lengths adding up to one; this is \cite[Theorem~2]{Moore3}.

\section{The weak equivalence before the globe is attached}
\label{sec:3}

For topological spaces \(K\) and \(Z\), let
\[
 \TOP_{\mathrm{co}}(K,Z)=\{f: K\longrightarrow Z\mid f\text{ is continuous}\}
\]
endowed with the compact-open topology.  Thus a subbasis of this topology consists of the subsets
\[
 [K_0,U]=\{f\in \TOP_{\mathrm{co}}(K,Z)\mid f(K_0)\subset U\},
\]
where \(K_0\subset K\) is compact and \(U\subset Z\) is open. Since \(I\) is compact, the compact-open topology on \(\TOP_{\mathrm{co}}(I,Z)\) is the topology of uniform convergence whenever \(Z\) is metrizable; see \cite[Proposition~A.13]{MR1867354}.  

The space \(\TOP_{\mathrm{co}}(I,I)\) is a convex metrizable subspace of the normed vector space \(\TOP_{\mathrm{co}}(I,\mathbb R)\). It is therefore locally path-connected and hence \(\Delta\)-generated by
\cite[Proposition~3.11]{MR3270173}.  Consequently, \(\TOP(I,I)=\TOP_{\mathrm{co}}(I,I)\). Let
\[
 \mathcal V=\{\gamma\in \TOP_{\mathrm{co}}(I,I)\mid \gamma(0)=\gamma(1)=0\},
\]
which is given the subspace topology.  Explicitly, its topology is induced by the uniform metric
\[
 d_\infty(\gamma,\eta)=\max_{t\in I}\lvert\gamma(t)-\eta(t)\rvert.
\]
Put
\[
 m(\gamma)=\max_{t\in I}\gamma(t),\qquad
 \mathcal E_1=m^{-1}(1),\qquad
 \mathcal E_{>0}=\mathcal V\setminus\{0\}.
\]
The function \(m\) is continuous because
\begin{equation}\label{eq:max-lipschitz}
 \lvert m(\gamma)-m(\eta)\rvert
 \leq d_\infty(\gamma,\eta).
\end{equation}

Both \(\mathcal E_1\) and \(\mathcal E_{>0}\) are \(\Delta\)-generated.  Here is a verification.  The space \(\mathcal V\) is convex in the normed vector space \(\TOP_{\mathrm{co}}(I,\mathbb R)\), and \(\mathcal E_{>0}\) is open in \(\mathcal V\); hence it is metrizable and locally path-connected.  The inequality~\eqref{eq:max-lipschitz} proves that \(m\) is
continuous.  Together with the continuity of scalar multiplication in
\(\TOP_{\mathrm{co}}(I,\mathbb R)\), this shows that the two maps
\[
\mathcal E_{>0}\longrightarrow \mathcal E_1\times(0,1],
\qquad
\gamma\longmapsto
\left(\frac{\gamma}{m(\gamma)},m(\gamma)\right),
\]
and
\[
\mathcal E_1\times(0,1]\longrightarrow\mathcal E_{>0},
\qquad
(\eta,r)\longmapsto r\eta,
\]
are continuous.  A direct calculation shows that they are inverse
homeomorphisms. Here \(\mathcal E_1\times(0,1]\) is equipped with the ordinary product topology. Since \(\mathcal E_{>0}\cong\mathcal E_1\times(0,1]\) is locally path-connected and the projection onto \(\mathcal E_1\) is open, \(\mathcal E_1\) is locally path-connected.  It is metrizable as a subspace of \(\mathcal V\).  A first-countable locally path-connected space is \(\Delta\)-generated by \cite[Proposition~3.11]{MR3270173}.  Consequently the \(\Delta\)-kelleyfication used in the definition of the execution-path topology does not alter either of these two spaces.  

Define two \(\Pcat\)-multipointed \(d\)-spaces on the same underlying multipointed space
\[
 (I,\{0\})
\]
by
\[
 \Path A=\mathcal E_1,
 \qquad
 \Path B=\mathcal E_{>0}.
\]
These are well-defined \(\Pcat\)-multipointed \(d\)-spaces. Since every \(\phi\in\Pcat(1,1)\) is surjective, \(\gamma\phi\) has the same image, and hence the same maximum, as \(\gamma\), for every \(\gamma\in \mathcal V\). Also, the image of a normalized composite of two based loops is the union of their images.  Thus \(\mathcal E_1\) and \(\mathcal E_{>0}\) are both closed under the two required operations, and all their elements are nonconstant. 

The identity of \(I\) induces a map
\begin{equation}\label{eq:s}
 s: A\longrightarrow B
\end{equation}
because \(\mathcal E_1\subset\mathcal E_{>0}\).

\begin{lemma}\label{lem:contractible}
The spaces \(\mathcal E_1\) and \(\mathcal E_{>0}\) are contractible. Consequently \(s\) is a weak equivalence.
\end{lemma}

\begin{proof}
Let \(b(t)=4t(1-t)\).  Thus \(b\in\mathcal E_1\).  A contraction of \(\mathcal E_{>0}\) to \(b\) is
\[
 H_u(\gamma)=(1-u)\gamma+ub.
\]
For \(u>0\), one has \(H_u(\gamma)(1/2)\geq u\), whereas \(H_0(\gamma)=\gamma\neq 0\).  Hence the homotopy never meets the zero loop.

A contraction of \(\mathcal E_1\) to \(b\) is
\[
 K_u(\gamma)=
 \frac{(1-u)\gamma+ub}{m((1-u)\gamma+ub)}.
\]
The denominator is positive by the preceding observation, and \eqref{eq:max-lipschitz} proves continuity.  Every value of \(K\) is a based loop with maximum one.  Thus the two displayed formulas are indeed contractions.

The map \(s\) is the identity on the singleton state set. The induced map on execution-path spaces has nonempty contractible source and target.  It therefore induces a bijection on path components and isomorphisms on all homotopy groups. It is a weak homotopy equivalence, and the characterization of weak equivalences recalled above proves the last assertion.
\end{proof}

Iterated normalized compositions produce execution paths with arbitrarily many intermediate returns to \(0\). These intermediate returns do not affect the contractibility argument of Lemma~\ref{lem:contractible}.  For \(u>0\) and \(0<t<1\),
	\[
	(1-u)\gamma(t)+ub(t)\geq ub(t)>0.
	\]
Thus, at every positive time of the contraction, all intermediate visits to \(0\) have disappeared.  The same observation applies to the contraction of \(\mathcal E_1\) after division by the maximum.  Hence the set of intermediate times at which an execution path passes through \(0\) is not a homotopy invariant and creates no obstruction to contractibility.

\begin{remark}\label{rem:saturated}
All \(\G\)-multipointed \(d\)-spaces are saturated by \cite[Proposition~28]{Moore3}. For \(\Pcat=\M\), both \(A\) and \(B\) are saturated in the sense of \cite[Section~8]{Moore3}.  Indeed, if \(\gamma\phi\in\mathcal E_1\) (respectively \(\gamma\phi\in\mathcal E_{>0}\)) for some \(\phi\in\M(1,1)\), then surjectivity of \(\phi\) gives \(\operatorname{Im}(\gamma\phi)=\operatorname{Im}(\gamma)\).  Hence \(m(\gamma)=1\) (respectively \(\gamma\neq0\)), so \(\gamma\) is already an execution path.  The pushouts below are calculated in the category of all \(\M\)-multipointed \(d\)-spaces; no assertion about pushouts calculated in the reflective subcategory of saturated objects is needed here.
\end{remark}

\section{The globular pushout}
\label{sec:4}

Take the generating q-cofibration \eqref{eq:generator} for \(n=0\):
\begin{equation}\label{eq:i}
 i:\Glob(\varnothing)\longrightarrow\Glob(\mathsf{D}^0).
\end{equation}
The source has two states and no execution paths.  Map both of its states to the unique state \(0\) of \(A\), and use the composite with \(s\) as the attaching map to \(B\).  Form the two pushouts
\[
\begin{tikzcd}[column sep=4.2em,row sep=3em]
 \Glob(\varnothing) \arrow[r] \arrow[d,"i"',rightarrowtail]
   & A \arrow[d,rightarrowtail,"j"'] \arrow[r,"s"] & B \arrow[d,rightarrowtail] \\
 \Glob(\mathsf{D}^0) \arrow[r] & \cocartesian X \arrow[r,"\bar s"] & \cocartesian Y .
\end{tikzcd}
\]
The map \(j:A\to X\) is a q-cofibration, since it is a pushout of the generating q-cofibration \(i\).

The underlying spaces of \(X\) and \(Y\) are the same compact metrizable space
\begin{equation}\label{eq:wedge}
 W=I_{\mathrm{old}}\vee \mathsf{S}^1_{\mathrm{new}}.
\end{equation}
Indeed, the underlying space of \(\Glob(\mathsf{D}^0)\) is an interval \(I_{\mathrm{new}}\), and the attaching map identifies both endpoints with \(0\in I_{\mathrm{old}}\). The space \(W\) is a finite CW complex.  It is therefore Hausdorff, first countable, and locally path-connected, and hence \(\Delta\)-generated by \cite[Proposition~3.11]{MR3270173}. Thus \(W\) already belongs to \(\Top\), and the pushout is unchanged when calculated in \(\Top\). The map \(\bar s\) is the identity on \(W\) and on the singleton state set; only their execution-path structures differ.

Let
\[
 c: I\longrightarrow I_{\mathrm{new}}/(0\sim 1) \cong \mathsf{S}^1_{\mathrm{new}}\subset W,
 \qquad c(t)=[t]
\]
be the execution path supplied by the new globe, traversing the new circle once.  Regard the loop \(b(t)=4t(1-t)\) as an old execution path.  Thus both \(c\) and the normalized composite \(c*_N b\) are execution paths of \(X\), as well as of \(Y\).

The strategy of Lemma~\ref{lem:separated} and Lemma~\ref{lem:connected} is as follows. In \(X\), every old execution path reaches height one, which produces the discrete invariant \(\Lambda\); in \(Y\), the additional old paths \(b_u\) have arbitrarily small positive height, allowing that invariant to disappear in a continuous family.

\begin{lemma}\label{lem:separated}
The paths \(c\) and \(c*_N b\) belong to distinct path components of \(\Path X\).
\end{lemma}

\begin{proof}
Define a continuous map \(h: W\to I\) by
\[
 h(x)=x\quad(x\in I_{\mathrm{old}}),
 \qquad
 h(x)=0\quad(x\in \mathsf{S}^1_{\mathrm{new}}).
\]
The two formulas agree at the wedge point, so continuity follows from the quotient universal property of \eqref{eq:wedge}.  Define
\[
 \Lambda:\Path X\longrightarrow I,
 \qquad
 \Lambda(\gamma)=\max_{t\in I}h(\gamma(t)).
\]
Postcomposition by \(h\) is continuous on compact-open mapping spaces, and the maximum functional is continuous by \eqref{eq:max-lipschitz}; after \(\Delta\)-kelleyfication the same map remains continuous by functoriality of \(\Delta\)-kelleyfication.

Every old execution path of \(A\) has maximum one after postcomposition by \(h\), whereas every execution path supplied by the new globe has maximum zero.  By \cite[Theorem~2]{Moore3}, every execution path of the pushout \(X\) is a finite Moore composition of such paths.  The maximum of a composite is the maximum of the maxima of its factors, and a surjective reparametrization does not change the image.  Consequently
\[
 \Lambda(\Path X)\subset\{0,1\}.
\]
The two fibres of \(\Lambda\) are therefore open and closed in \(\Path X\). Since \(\Lambda(c)=0\) and \(\Lambda(c*_N b)=1\), the two paths cannot belong to the same path component.
\end{proof}

\begin{lemma}\label{lem:connected}
The paths \(\bar s(c)=c\) and \(\bar s(c*_N b)=c*_N b\) belong to the same path component of \(\Path Y\).
\end{lemma}

\begin{proof}
For \(u>0\), let \(b_u(t)=u b(t) = 4ut(1-t)\).  It is a nonconstant based loop and hence an execution path of \(B\).  Put \(\ell_u=1-u/2\).  Define \(F: I\to\Path Y\) by \(F(0)=c\) and, for \(u>0\),
\begin{equation}\label{eq:F}
 F(u)(t)=
 \begin{cases}
  c(t/\ell_u),&0\leq t\leq\ell_u,\\[2mm]
  b_u\bigl((t-\ell_u)/(1-\ell_u)\bigr),
     &\ell_u\leq t\leq1.
 \end{cases}
\end{equation}
For \(u>0\), this is an execution path for either choice of \(\Pcat\). Indeed, it is the normalized composite \(c*_N b_u\) precomposed with the piecewise linear increasing homeomorphism
\[
 \theta_u(t)=
 \begin{cases}
  t/(2\ell_u),&0\leq t\leq\ell_u,\\[1mm]
  \dfrac12+(t-\ell_u)/u,&\ell_u\leq t\leq1,
 \end{cases}
\]
which belongs to \(\G(1,1)\subset\Pcat(1,1)\).

It remains to justify continuity at \(u=0\).  Consider the adjoint set map \(\widehat F: I\times I\to W\).  On the closed subset \(\{(u,t)\mid t\leq\ell_u\}\), the first formula in \eqref{eq:F} is continuous because \(\ell_u\geq1/2\).  On the closed subset \(\{(u,t)\mid t\geq\ell_u\}\), the second formula extends continuously to its only point with \(u=0\), which is \((0,1)\), by assigning the wedge point to it: its old coordinate is bounded by \(u\) and therefore tends to zero.  The two formulas agree at \(t=\ell_u\).  The pasting lemma proves that \(\widehat F\) is continuous: for every closed subset of \(W\), its inverse image is the union of two subsets closed in the two closed pieces, and is therefore closed in \(I\times I\). Cartesian closedness of \(\Top\) now yields a continuous transpose
\[
I\longrightarrow \TOP(I,W).
\]
This map takes its values in the set \(\Path Y\).  Since \(I\) is \(\Delta\)-generated, the universal property of \(\Delta\)-kelleyfication shows that \(F: I\to\Path Y\) is continuous for the prescribed execution-path topology.

Finally \(F(0)=c\), while \(\ell_1=1/2\), \(b_1=b\), and hence \(F(1)=c*_N b\).  Thus \(F\) is the required path in \(\Path Y\).
\end{proof}

\begin{theorem}\label{thm:counterexample}
	For every \(\Pcat\in\{\G,\M\}\), the pushout map \(\overline{s}: X\to Y\) is not a weak equivalence.  Consequently, the q-model category of \(\Pcat\)-multipointed \(d\)-spaces is not left proper.  More precisely, left properness already fails for the cobase change \(A\to X\) of the single globular generating q-cofibration~\eqref{eq:i}; neither state generator \(C\) nor \(R\) is involved.
\end{theorem}

\begin{proof}
By Lemma~\ref{lem:separated}, the elements \(c\) and \(c*_N b\) determine distinct elements of \(\pi_0(\Path X)\).  Lemma~\ref{lem:connected} says that their images determine the same element of \(\pi_0(\Path Y)\). Therefore
\[
 \pi_0(\Path\bar s):
 \pi_0(\Path X)\longrightarrow\pi_0(\Path Y)
\]
is not injective.  The map \(\Path\bar s\) is not a weak homotopy equivalence.  Since \(\bar s\) is a bijection on states, the definition of the weak equivalences shows that \(\bar s\) is not a weak equivalence. The right-hand square of the diagram above is a pushout, and \(A\to X\) is a q-cofibration.  Since \(s\) is a weak equivalence, this proves that the q-model structure is not left proper. The cell attachment uses the globular generator~\eqref{eq:i}, rather than either of the state generator \(C\) or \(R\), which proves the final assertion.
\end{proof}


\end{document}